\documentclass[12pt]{amsart}
\usepackage{ae,amsfonts,euscript,enumerate}
\usepackage{amsmath,euscript}
\usepackage{amsthm}
\usepackage{comment}
\usepackage{amssymb,array}
\usepackage{fullpage}

\usepackage[usenames]{color}

\usepackage[colorlinks=true,
linkcolor=webgreen,
filecolor=webbrown,
citecolor=webgreen]{hyperref}

\definecolor{webgreen}{rgb}{0,.5,0}
\definecolor{webbrown}{rgb}{.6,0,0}

\usepackage{url}

\usepackage{float}

\usepackage{graphics}
\usepackage{latexsym}
\usepackage{epsf}

\newtheorem{theorem}{Theorem}[section]
\newtheorem{lemma}[theorem]{Lemma}
\newtheorem{corollary}[theorem]{Corollary}

\newtheorem{proposition}[theorem]{Proposition}

\theoremstyle{definition}
\newtheorem{example}[theorem]{Example}
\newtheorem{remark}[theorem]{Remark}

\newcommand{\Enn} {\mathbb N}

\def\modd#1 #2{#1\ ({\rm mod}\ #2)}

\begin{document}

\author{Jason Bell}
\address[J. Bell and K. Hare]{Department of Pure Mathematics \\
University of Waterloo\\
Waterloo, ON  N2L 3G1 \\
Canada}
\email[J. Bell]{jpbell@uwaterloo.ca}
\thanks{Research of the first author supported by NSERC Grant 2016-03632.}
\author{Kathryn Hare}
\email[K.~Hare]{kehare@uwaterloo.ca}
\thanks{Research of the second author supported by NSERC Grant 2016-03719.}

\author{Jeffrey Shallit}
\address[J. Shallit]{School of Computer Science \\
University of Waterloo\\
Waterloo, ON  N2L 3G1 \\
Canada}
\email[J.~Shallit]{shallit@uwaterloo.ca}
\thanks{Rsearch of the third author supported by NSERC Grant 105829/2013.}

\title{When is an automatic set an additive basis?}

\begin{abstract}
We characterize those $k$-automatic sets $S$ of natural numbers that form an
additive basis for the natural numbers, and we show that this characterization is
effective.  In addition, we give an algorithm to
determine the smallest $j$ such that $S$ forms an additive basis of
order $j$, if it exists.
\end{abstract}

\maketitle

\section{Introduction}
One of the principal problems of additive number theory is to determine,
given a set $S \subseteq \Enn$, whether
there exists a constant $j$ such that every natural
number (respectively, every sufficiently large natural number)
can be written as a sum of at most $j$ members of $S$ 
(see, e.g., \cite{Nathanson:1996}).
If such a $j$ exists, we say that $S$ is an {\it additive basis} (resp.,
an {\it asymptotic additive basis}) {\it of order $j$} for $\Enn$.

Variants of this problem date back to antiquity, with Diophantus asking whether every natural number could be expressed as a sum of four squares.  More generally, Waring's problem asks whether the set of $k$-th powers forms an additive basis for the natural numbers, which was ultimately answered in the affirmative by Hilbert \cite[Chapter 3]{Nathanson:1996}.  The problem of finding bounds on the number of $k$-th powers required to express all natural numbers and all sufficiently large natural numbers, as well as 
whether restricted subsets of $k$-th powers form additive bases,
continues to be an active area of research \cite{Vaughan&Wooley:1991,Wooley:1992,Wei&Wooley:2015}.  

Independent of Hilbert's work on Waring's problem, the famed 
Goldbach conjecture asks whether every even positive integer can be expressed as the sum of at most two prime numbers.  If true, this would then imply that every sufficiently large natural number is the sum of at most three prime numbers.  
Vinogradov \cite[Chapter 8]{Nathanson:1996}
has shown that every sufficiently large natural number can be expressed as the sum of at most four prime numbers, and so the set of prime numbers is an asymptotic additive basis for the natural numbers.  

From these classical beginnings,
a general theory of additive bases has since emerged,
and the problem of whether given sets of natural numbers form additive bases (or asymptotic additive bases) has been considered for many classes of sets.

If one adopts a computational point of view, subsets of natural numbers can be divided into two classes: computable sets (i.e., sets that can be produced using a Turing machine) and those sets that lie outside of the realm of classical computation.  Historically, the explicitly-given sets for which the problem of being an additive basis has been considered are computable, and a natural problem is to classify the computable subsets of the natural numbers that form additive bases.
However, a classical theorem
of Kreisel, Lacombe, and Shoenfield \cite{Kreisel&Lacombe&Shoenfield:1959}
implies that the question of whether
a given computable subset of $\Enn$ forms an additive basis
is, in general,
recursively unsolvable.
Even for relatively simple sets, the problem seems
intractable, as it applies to many sets of natural numbers, such as the set of twin primes, for which it is still open as to whether it is infinite, let alone whether it is an additive basis, which heuristics indicate should be the case \cite{Z:1979}.  
Thus it is of interest to identify some classes of sets for which the problem
is decidable.

One mechanism for producing computable sets is to fix a natural number $k\ge 2$ and consider natural numbers in terms of their base-$k$ expansions. A set of natural numbers can then be regarded as a sublanguage of the collection of words over the alphabet $\{0,1,\ldots ,k-1\}$.  In this setting, there is a coarse hierarchy, formulated by Chomsky, that roughly divides complexity into four nested classes: recursively enumerable languages (those that are produced using Turing machines); context-sensitive languages (those produced using linear-bounded non-deterministic Turing machines); context-free languages (those produced using pushdown automata); and regular languages (those produced using finite-state automata).  The simplest of these four classes is the collection of regular languages.  When one uses a regular sublanguage of the collection of words over $\{0,1,\ldots ,k-1\}$, the corresponding collection of natural numbers one obtains is called a $k$-\emph{automatic set} (see, for example, \cite{Allouche&Shallit:2003}).

In this paper we completely characterize those $k$-automatic sets
of natural numbers that form an additive basis or an asymptotic additive basis. In the case of a $k$-automatic set $S$ of natural numbers, there is a well-understood dichotomy: either $\pi_S(x):=\#\{n\le x\colon n\in S\}$ is ${\rm O}((\log\, x)^d)$ for some natural number $d$, or there is a real number $\alpha>0$ such that $\pi_S(x) = \Omega(x^\alpha)$ (see Section \ref{sec:basics} and specifically Corollary \ref{cor:sparse} for details).  In the case where $\pi_S(x)$ is asymptotically bounded by a power of $\log\, x$, we say that $S$ is \emph{sparse}.  Our first main result is the following theorem (see Theorem \ref{thm:main} and the remarks that follow).

\begin{theorem} Let $k\ge 2$ be a natural number 
and let $S$ be a $k$-automatic subset of $\Enn$.  Then $S$ forms an asymptotic additive basis for $\Enn$ if and only if the following conditions both hold:
\begin{enumerate}
\item $S$ is not sparse;
\item $\gcd(S)=1$.
\end{enumerate}
Moreover, if $S$ is a non-sparse set and $\gcd(S) = 1$,
then there exist effectively computable constants $M$ and $N$
such that every natural number greater than or equal to $M$ 
can be expressed as the sum of at most $N$ elements of $S$.
\label{thm:intro}
\end{theorem}
We note that a necessary condition for a set $S$ to be an additive basis is that $1$ be in $S$.  If $S$ is not sparse and $\gcd(S) = 1$ and $1\in S$, 
then $S$ is an additive basis, and these conditions are necessary.  We give 
explicit upper bounds on $M$ and $N$ in terms of the number of states in the minimal automaton that accepts the set $S$, and we show that these bounds are in some sense the correct form for the type of bounds one expects to hold in general.  
An interesting feature of our proof is that it uses results dealing with sums of Cantor sets obtained by the second-named author in work with Cabrelli and Molter 
\cite{Cabrelli&Hare&Molter:1997}.

Our second main result is the following.
\begin{theorem}
Let $k\ge 2$ be a natural number and let 
$S$ be a $k$-automatic subset of $\Enn$.
There is an algorithm that determines whether the conditions of
Theorem~\ref{thm:intro} hold, and if so, also determines
the smallest possible $N$ in that theorem and the corresponding smallest
possible $M$.
\label{thm:intro2}
\end{theorem}

The outline of this paper is as follows.  In Section \ref{sec:basics} we recall some of the basic concepts from the theory of regular languages and automatic sets---including the notion of a sparse automatic set---which play a key role in the statement of Theorem \ref{thm:intro}.  In Section \ref{sec:Cantor} we give some of the necessary background on Cantor sets and prove a key lemma involving these sets. In Section \ref{sec:main} we prove a strengthening of Theorem \ref{thm:intro} (see Theorem \ref{thm:main}) that gives explicit bounds on $M$ and $N$ appearing in the statement of the theorem.  In Section \ref{sec:algorithm}, we give an algorithm that allows one to find optimal bounds for given automatic sets and in Section \ref{sec:exam}, we give some examples to illustrate the usage of our algorithm.

\section{Basics}
\label{sec:basics}
We are concerned with words and numbers.  A {\it word} is a finite string of symbols over a finite alphabet $\Sigma$.   If $x$ is a word, then $|x|$ denotes its length (the number of symbols in it).  The {\it empty word} is the unique word of length $0$, and it is denoted by $\epsilon$.  

The {\it canonical base-$k$ expansion} of a natural number $n$ is the unique word over the alphabet $\Sigma_k = \{ 0,1,\ldots, k-1\}$ representing $n$ in base $k$, without leading zeros, starting with the most significant digit.  It is denoted $(n)_k$.  Thus, for example, $(43)_2 = 101011$.  If $w$ is a word,
possibly with leading zeros, then $[w]_k$ denotes the
integer that $w$ represents in base $k$.

A {\it language} is a set of words.  Three important languages are
\begin{itemize}
\item[(i)]  $\Sigma^*$, the set of all finite words over the alphabet $\Sigma$;
\item[(ii)]  $\Sigma^n$, the set of words of length $n$; and
\item[(iii)] $\Sigma^{\leq n}$, the set of words of length $\leq n$.
\end{itemize}
Given a set
$S\subseteq \Enn$, we write $(S)_k$ for the language
of canonical base-$k$ expansions of elements of $S$.

There is an ambiguity that arises from the direction in which base-$k$
expansions are read by an automaton.  In this article we always assume
that these expansions are read starting with the least significant digit.

We recall the standard asymptotic notation for functions from $\Enn$ to $\Enn$:
\begin{itemize}
\item $f = O(g)$ means that there exist constants $c> 0$, $n_0 \geq 0$ such that
$f(n) \leq c g(n)$ for $n \geq n_0$;
\item $f = \Omega(g)$ means that there exist constants $c> 0$, $n_0 \geq 0$ such that
$f(n) \geq c g(n)$ for $n \geq n_0$;
\item $f = \Theta(g)$ means that $f = O(g)$ and $f = \Omega(g)$.
\end{itemize}

Given a language $L$ defined over an alphabet $\Sigma$, its {\it growth function} $g_L (n)$ is defined to be $|L \ \cap \ \Sigma^n|$, the number of words in $L$ of length $n$.  If there exists a real number $\alpha > 1$ such that
$g_L (n) > \alpha^n$ for infinitely many $n$, then we say that
$L$ has {\it exponential growth}.  If there exists
a constant $c \geq 0$ such that
$g_L (n) = O(n^c)$, then we say that $L$ has
{\it polynomial growth}. 

A {\it deterministic finite automaton} or DFA is a quintuple
$M = (Q, \Sigma, \delta, q_0, F)$, where $Q$ is a finite nonempty
set of states, $\Sigma$ is the input alphabet, 
$q_0$ is the initial state, $F \subseteq Q$ is a set of
final states, and $\delta:Q \times \Sigma \rightarrow Q$
is the transition function.    The function $\delta$
can be extended to $Q \times \Sigma^* \rightarrow Q$ in the
obvious way.  The language accepted by $M$ is defined
to be $\{ x \in \Sigma^* \ : \ \delta(q_0, x) \in F\}$.
A language is said to be {\it regular} if there is a DFA
accepting it
\cite{Hopcroft&Ullman:1979}. 

A {\it nondeterministic finite automaton} or NFA is like a DFA, except
that the transition function $\delta$ maps $Q \times \Sigma$ to
$2^Q$.  A word $x$ is accepted if some path labeled $x$ causes the
NFA to move from the initial state to a final state.

We now state three well-known results about the growth functions of regular languages.  These lemmas follow by combining the results in, e.g.,
\cite{Ginsburg&Spanier:1966,Trofimov:1981,Ibarra&Ravikumar:1986,Szilard&Yu&Zhang&Shallit:1992,Gawrychowski&Krieger&Rampersad&Shallit:2010}.

\begin{lemma}
Let $L$ be a regular language.  Then $L$ has either polynomial
or exponential growth.  
\label{lem1}
\end{lemma}

Define $h_L (n) = |L \ \cap \ \Sigma^{\leq n}|$, the number of words of length $\leq n$.

\begin{lemma}  Let $L$ be a regular language.
The following are equivalent:
\begin{itemize}
\item[(a)]  $L$ is of polynomial growth;
\item[(b)]  there exists an integer
$d \geq 0$ such that $h_L (n) = \Theta(n^d)$;
\item[(c)]  $L$ is the finite union of languages of the form $z_0 x_1^* z_1 x_2^* \cdots z_{i-1} x_i^* z_i$ for words
$z_0, z_1, \ldots, z_i$, $x_1, x_2, \ldots x_i$;
\item[(d)]  there exist a constant $j$ and words
$y_1, y_2, \ldots, y_j$ such that
$L \subseteq y_1^* y_2^* \cdots y_j^*$.
\end{itemize}
\label{lem2}
\end{lemma}

\begin{lemma}   Let $L$ be a regular
language, accepted by a DFA or NFA $M = (Q, \Sigma,\delta, q_0,F)$.  The following are equivalent:
\begin{itemize}
\item[(a)]  $L$ is of exponential growth;
\item[(b)]  there exists a real number
$\alpha > 1$ such that $h_L (n) = \Omega(\alpha^n)$;
\item[(c)]  there exists a state $q$ of $M$ and words $w_0, x_0,x_1, z_0$
such that $x_0 x_1 \not= x_1 x_0$ and $\delta(q_0,w_0) =
\delta(q,x_0) = \delta(q,x_1) = q$, and
$\delta(q,z_0) \in F$;
\item[(d)]   there exist words $w,x,y,z$
with $xy \not= yx$
such that $w \{x,y\}^* z \subseteq L$;
\item[(e)] there exist words $s,t,u,v$
with $|t| = |u|$ and $t \not= u$
such that $s \{t,u\}^* v \subseteq L$.
\end{itemize}
\label{lem3}
\end{lemma}

We will also need the following result, which appears to be new.

\begin{lemma}
In Lemma~\ref{lem3} (e), the words $s, t, u, v$ can be
taken to obey the following inequalities:  
$|s|, |v| <n$ and $|t|, |u| < 3n$, where $n$ is the number of states
in the smallest DFA or NFA $M$ accepting $L$. 
\label{lem4}
\end{lemma}

\begin{proof}

Consider those quadruples of words 
$(w_0, x_0, x_1, z_0)$ satisfying the conditions
of Lemma~\ref{lem3} (c), namely, that
there is a state $q$ of $M$ such that
$\delta(q_0, w_0) = \delta(q, x_0) = \delta(q,x_1) = q$, and
$\delta(q,z_0) \in F$, and
$x_0 x_1 \not= x_1 x_0$.
We can choose $w_0$ and $z_0$ minimal
so that no state is encountered more than once via the paths 
$P_{w_0}$ and $P_{z_0}$ through $M$
labeled $w_0$ and $z_0$, respectively.  Thus without loss
of generality we can assume $|w_0|,|z_0|<n$. 

Next, among all such $x_0, x_1$, assume $x_0$ is a shortest
nonempty word and $x_1$ is a shortest nonempty word paired with
$x_0$.  Consider the set of states encountered when
going from $q$ to $q$ via the path $P_{x_0}$ labeled $x_0$.
If some state (other than $q$) is
encountered twice or more, 
this means there is a loop we can cut out and find
a shorter nonempty word $x'_0$ with $\delta(q,x'_0)=q$.
By minimality of the length of $x_0$,
we must have that $x'_0$ commutes with all words $w$ 
such that $\delta(q,w)=w$.
In particular, $x'_0$ commutes with $x_0$ and $x_1$.
Since the collection of words that commute with a non-trivial
word consists of powers of a common word \cite[Prop.~1.3.2]{Lothaire:1997},
we see that if this were the case, then
$x_0$ and $x_1$ would commute, a contradiction.
Thus $|x_0| \leq n$.  By construction $|x_1|\ge |x_0|$.
If $x_0$ is a proper prefix of $x_1$, then we have $x_1=x_0 x'_1$
for some nonempty word $x'_1$ with $\delta(q,x'_1)=q$,
and since $x_0 x_1\neq x_1 x_0$,
we have $x_0 x_0 x'_1\neq x_0 x'_1 x_0$.
Cancelling $x_0$ on the left gives $x_0 x'_1\neq x'_1 x_0$.
But this contradicts minimality of the length of $x_1$.  

Thus $x_1$ has some prefix $p$ with
$|p|\le |x_0|$ such that $x_1=p p' $ and $p$ is not a prefix of $x_0$.
Let $q'=\delta(q,p)$.  If $q'=q$ then we have $\delta(q,p)=q$ 
and $xp\neq px$ since $p$ is not a prefix of $x$.
Thus in this case, by minimality of $x_1$, we have $x_1=p$ and so $|x_1|\le n$.
Thus we may assume that $q'\neq q$.  Then $\delta(q',p')=q$.
Let $u$ be the label of a shortest path from $q'$ to $q$.
Then $|u|<n$ since by removing loops, we may assume the path $P_u$
visits no state more than once and it does not revisit $q'$.
Observe that $|pu|<2n$ and $\delta(q,pu)=q$.
Moreover, $x pu \neq pux$ since $p$ is not a prefix of $x$.
Thus, by the minimality of $x_1$, we have $|x_1|\le |up|<2n$.  

Thus we can assume that
$|x_0|\leq n$ and $|x_1| < 2n$.
Setting $s = w_0$, $t = x_0x_1$, $u = x_1 x_0$, and $v = z_0$
gives the desired inequalities.  
\end{proof}

\begin{remark}
The bound $3n-1$ in Lemma~\ref{lem4} is optimal.
For example, consider an NFA $M = (\{q_1, \ldots, q_n\},
\{ a,b\} , \delta, q_1, \{q_1 \} )$
with $n$ states
$q_1, q_2, \ldots, q_n$ connected in a directed cycle
with transitions labeled by $a$.    Add a directed edge
labeled $b$ from $q_n$ back to $q_2$.  Then
the smallest words obeying the conditions
are $x = a^n$ of length $n$ and $y = a^{n-1} b a^{n-1}$ of length $2n-1$.
Then $t = xy$ and $u = yx$ and $|t| = |u| = 3n-1$.
\end{remark}

\begin{theorem}
Given a regular language represented by a DFA or NFA, we can decide in linear time whether the language has polynomial or exponential growth.
\label{thm1}
\end{theorem}

\begin{proof}
See, for example, \cite{Gawrychowski&Krieger&Rampersad&Shallit:2010}.
\end{proof}

Now let us change focus to sets of integers.  
Given a subset $S\subseteq \mathbb{N}$ we define \begin{equation}
\pi_S(x)=\{n\le x\colon n\in S\}.\end{equation}    If there exists an
integer $d \geq 0$ such that $\pi_S (x) = O( (\log x)^d)$, then we say that $S$ is {\it sparse}.  Otherwise we
say $S$ is {\it non-sparse}.

Then the corollary below follows immediately from the above results.

\begin{corollary}
\label{cor:sparse}
Let $k \geq 2$ be an integer and $S$ be a $k$-automatic
subset of $\Enn$.  
Then $S$ is non-sparse iff there exists a real number $\alpha>0$ such that $\pi_S(x) = \Omega(x^\alpha)$.
\end{corollary}

Given sets $S,T$ of real numbers, we let $S+T$ denote the
set $$\{ s+t \ : \ s \in S, t \in T \}.$$  Furthermore, we let
$S^j = \overbrace{S + S + \cdots + S}^j$; this is called the
$j$-fold sum of $S$.  We let
$S^{\leq j} = \bigcup_{1 \leq i \leq j} S^i$.  Note that
$S^{\le j}$ and $S^j$ denote,
respectively,
the set of numbers that can be written as a sum of at most $j$ elements of $S$,
and those that can be written as a sum of exactly $j$ elements of $S$.
Finally, if $S$ is a set of real numbers and $\alpha$ is a real number,
then $\alpha S = \{ \alpha x  \ : \ x \in S \}$.

\section{Sums of Cantor sets}
\label{sec:Cantor}
In this section, we quickly recall the basic notions we will make use of concerning Cantor sets.  Specifically, we will be dealing with central Cantor sets, which we now define.  Let $(r_k)_{k\ge 1}$ be a sequence of real numbers
in the half-open interval $(0, {1 \over 2}]$.
Given real numbers $\alpha < \beta$,
we define a collection of closed intervals $\{C_w\ \ : w\in \{0,1\}^*\}$,
where each $C_w \subseteq [\alpha,\beta]$, inductively as follows. 
We begin with $C_\epsilon =[\alpha,\beta]$.  Having defined $C_w$ for all binary words of length at most $n$, given a word $w$ of length $n+1$, we write $w=w' a$ with $|w'|=n$ and $a\in \{0,1\}$.  If $a=0$, we define $C_w$ to be the closed interval uniquely defined by having the same left endpoint as $C_{w'}$ and satisfying $|C_{w}|/|C_{w'}|=r_{n+1}$; If $a=1$, we define $C_w$ to be the closed interval uniquely defined by having the same right endpoint as $C_{w'}$ and satisfying $|C_{w}|/|C_{w'}|=r_{n+1}$.  We then take $C_n$ to be the union of the $C_w$ as $w$ ranges over words of length $n$.  It is 
straightforward to see that 
$$C_0\supseteq C_1\supseteq C_2\supseteq \cdots,$$ 
and the intersection of these sets is called the \emph{central Cantor set} 
associated with the ratios $r_k$ and initial interval $[\alpha, \beta]$.  The associated real numbers $r_k$ are called the associated \emph{ratios of dissection}, and in the case when there is a fixed $r$ such that $r_k=r$ for every $k\ge 1$, we simply call $r$ the ratio of dissection.  A key example is the classical ``middle thirds'' Cantor set, which is the central Cantor set with ratio of dissection ${1 \over 3}$ and initial interval $[0,1]$.

Let $k\ge 2$ be a natural number and let $u,y,z\in \Sigma_k^*$ with $|y| = |z|$ and $y\neq z$.  In particular, $y$ and $z$ are nonempty. We define
$C(u;y,z)$ to be the collection of real numbers whose base-$k$ expansion is of the form
$0.u w_1w_2w_3\cdots $ with each $w_i\in \{y,z\}$.  For example, when $k=3$, $u$ is the empty word, $y=0$, and $z=2$, $C(u;y,z)$ is the usual Cantor set. A key lemma used in our considerations rests on a result of Cabrelli, the second-named author, and Molter \cite{Cabrelli&Hare&Molter:1997}, which says that a set formed by taking the sum of $N$ elements from a Cantor set with a fixed ratio of dissection is equal to an interval when $N$ is sufficiently large.  We use this result to prove the following lemma.
\begin{lemma} Let $k\ge 2$ and $t\ge 1$ be natural numbers and let $u,y,z\in \Sigma_k^*$ with $|y| = |z|$ and $y\neq z$.  Suppose that $|u|=L$ and $|x|=|y|=s$.  Then every real number $\gamma \in [k^{L+s+1}, k^{L+s+1+t}]$ can be expressed as a sum of at most $k^{2L+2s+t+1}$ elements from $C(u;y,z)$.
\label{lem: hare}
\end{lemma}
 \begin{proof} Let $s=|y|=|z|$ and write
$y=y_{1}\cdots y_{s}$, $z=z_{1}\cdots z_{s}$, and $u=u_{1}\cdots u_{L}$. 
Define 
\begin{align*}
Y &=\sum_{j=1}^{s}y_{j}k^{-j}\\
Z &=\sum_{j=1}^{s}z_{j}k^{-j}\\
U &=\sum_{j=1}^L u_j k^{-j}. 
\end{align*}
We may assume without loss of generality that $Y<Z$.
Consider the compact set $C=C(\epsilon ;y,z)$, the numbers whose base-$k$
expansion is of the form $0.x_{1}x_{2}x_{3}\cdot \cdot $ where $x_{i}\in
\{y,z\}$. The two
contractions, $S_{1}(x)=k^{-s}x+Y$ and $S_{2}(x)=k^{-s}x+Z$, clearly map $C$ into $C$, hence $C$ contains $S_{1}(C)\, \cup \, S_{2}(C)$. We claim that this containment is in fact an equality.  To see this, let $x$
be a real number with base-$k$ expansion $0.x_{1}x_{2}x_{3}\cdots$ with 
$x_{i}\in \{y,z\}$.  Then $x$ is mapped to $0.y x_{1}x_{2}\cdots$ under $S_{1}$ and to $0. z x_{1}x_{2}\cdot \cdot \cdot $ under $S_{2}$.  In particular, $x=S_1(0.x_2x_3\cdots)$ if $x_1=y$ and $x=S_2(0.x_2x_3\cdots)$ if $x_1=z$.

Next, consider $C'$, the set obtained by beginning with the
non-trivial interval $[\alpha ,\beta ]$ where $\alpha =(1-k^{-s})^{-1}Y$ and 
$\beta =(1-k^{-s})^{-1}Z$, and forming the central Cantor set with ratio of
dissection $k^{-s}$. 

Then $C'$ also has the property that $C'=S_1(C')\cup S_2(C')$.  Indeed, the set $C'_n$ that arises at level $n$ in the
Cantor set construction is the union of the images of $[\alpha ,\beta ]$ under the $n$-fold compositions $S_{j_{1}}\circ \cdots \circ S_{j_{n}}$,
where $j_{i}\in \{1,2\}$ for $i=1,\ldots ,n$.  Then $C'$ is simply the intersection of the $C'_n$ for $n\ge 1$.

Since there is a unique non-empty compact set with the above invariance property under the two
contractions $S_1$ and $S_2$, we must have $C=C'$. Thus $C$ has a central Cantor set
construction with ratio of dissection $k^{-s}$. It now follows
from \cite[Prop.~2.2]{Cabrelli&Hare&Molter:1997} 
that the $m$-fold sum $C^m$
equals the interval $[m\alpha ,m\beta ]$ 
whenever $m\geq k^{s}-1$.

The set $C(u;y,z)$ is equal to $\sum_{j=1}^{L}u_{j}k^{-j}+k^{-L}C:=U+k^{-L}C$.
Observe that if $C^m =[c,d],$ then $(k^{-L}C)^m =
[k^{-L}c,k^{-L}d]$ and the $m$-fold sum of $U+k^{-L}C$ is simply the interval $mU+[k^{-L}c,k^{-L}d]$. Thus for all $m\geq k^{s}-1$, $C(u;y,z)^m$ contains
the non-trivial interval $mI$ where $I=[U+k^{-L}\alpha ,U+k^{-L}\beta ]$.  The 
intervals $mI$ and $(m+1)I$ overlap whenever 
$$(m+1) (U+k^{-L}\alpha) \le m(U+k^{-L}\beta),$$ 
which occurs precisely when $m\ge (k^L U + \alpha) (\beta-\alpha)^{-1}$.  Since $\beta-\alpha\ge 1/k^s$ and $U,\alpha\le 1$, 
we see that for $m\ge k^{L+s}+k^s$, the intervals
$mI$ and $(m+1)I$ overlap.  Thus
$$\bigcup_{m\ge k^{L+s}+k^s} mI \supseteq [k^{L+s+1},\infty).$$

Consequently, we have 
that the interval $[k^{L+s+1},k^{L+s+1+t}]$ is contained in the union of the $m$-fold sums of $C(u;y,z)$ with $m= k^{L+s}+k^s,\ldots ,N$ whenever $N$ is such that
$N(U+k^{-L}\beta)\ge k^{L+s+t+1}$.  Since $U+k^{-L}\beta\ge k^{-L-s}$ we see that
we can take $N= k^{2L+2s+t+1}$.
This proves that every number in $
[k^{L+s+1},k^{L+s+1+t}-1]$ can be expressed as a sum of at most $N$ elements
from $C(u;y,z)$.
\end{proof}

\section{The first main result}
\label{sec:main}
In this section we prove the following theorem.

\begin{theorem} Let $k\ge 2$ be a natural number and let $S$ be a non-sparse $k$-automatic subset of $\mathbb{N}$ with $\gcd(S)=1$.  Then there exist effectively computable natural numbers $N=N(S)$ and $M=M(S)$ such that every natural number $n\ge M$ can be expressed as a sum of at most $N$ elements from $S$.  Moreover, if the minimal DFA accepting $S$ has $m$ states, then $N\le 5k^{16m+3}$ and $M\le 3k^{16m+5}$.  
\label{thm:main} 
\end{theorem}

\begin{remark}
We note that the non-sparse and gcd hypotheses on $S$ are, in fact,
necessary to obtain the conclusion of the statement of the theorem.  

If $\gcd(S) = g > 1$, then every sum of elements of $S$ is divisible by $g$.

On the other hand, if $S$ is a sparse $k$-automatic set then $\pi_S(x)=O((\log x)^d)$ for some $d\ge 0$.  
In particular, there is some $C>0$  such that for
all $x \geq 2$
there are at most $C(\log x)^d$
elements of $S$ that are $< x$.   Thus there are
at most $C^i(\log x)^{di} $ elements of $S$ smaller than
$x$ that can be
written as the sum of $i$ elements of $S$.
Hence there are at most $\sum_{0 \leq i \leq I} 
C^i (\log x)^{di}$ elements of $S$ smaller than
$x$ that can be written as the sum of at most
$I$ elements of $S$.  But this is
$O((\log x)^{dI + 1})$, which for large $x$
is smaller than $x$.  
\label{remark:main}
\end{remark}
This remark combined with Theorem \ref{thm:main} easily gives Theorem \ref{thm:intro}.
\begin{remark}
\label{remark:main2}

The bounds in Theorem~\ref{thm:main} are close to optimal.
If one considers the set $S$ of all natural numbers whose base-$k$
expansion has $j$ digits, for $j \geq 0$ and $j\equiv -1 $ (mod $m$),
then the minimal DFA accepting $S$ has size $m$.  On the other hand, every element of $S$ has size at least $k^{m-2}$. So for each natural number 
$d \geq 1$ the interval $[1,k^{md-2}-1]\cap S$ has size at most
$k^{m(d-1)-1}-1$. Thus $k^{md-2}-1$ cannot be expressed
as a sum of fewer than $k^{m-2}$ elements of $S$ for $m\ge 2$.  
\end{remark}

Before we prove Theorem~\ref{thm:main}, we need some auxiliary results.  We recall that a subset $T$ of the natural numbers is {\it $c$-syndetic} for a natural number $c$ if $n\in T$ implies that there exists $i\in \{1,\ldots ,c\}$ such that $n+i\in T$.  If $T$ is $c$-syndetic for some $c$, we say that $T$ is 
{\it syndetic}.

\begin{proposition}
Let $k\ge 2$ be a natural number and let $S$ be a non-sparse $k$-automatic subset of the natural numbers whose minimal accepting DFA has $m$ states.  If $T$ is the set of all numbers that can be written as a sum of at most $k^{11m+1}$ elements of $S$, then for each $M>k^{7m+1}$ there exists $n\in T$ such that 
$|M-n|< k^{12m+1}$.  In particular, $T$ is $(2k^{12m+1})$-syndetic. 
\label{prop: syndetic}
\end{proposition}

\begin{proof} Since $S$ is non-sparse, by Lemma~\ref{lem4} we have that there exist words $u,y,z,v\in \Sigma_k^*$ with $y\neq z$ and $|u|,|v|\le m$, $|y|=|z|\le 3m$ such that $\mathcal{L}(S)$ contains $u\{y,z\}^*v$.  Let $L=|u|$ and $s=|y|=|z|$.  By Lemma \ref{lem: hare}, taking $t=s$, each $\alpha\in [k^{L+s+1}, k^{L+2s+1}]$ can be expressed as a sum of at most $k^{2L+3s+1}\le k^{11m+1}$ elements from $C(u;y,z)$.

Now let $0\le \alpha<\beta<1$ be real numbers. Suppose that $M$ is a natural number with base-$k$ expansion $x_0x_1\cdots x_d$ (and $x_0\neq 0$) with $d\ge \max(L+2s+1,K+2L+s+2)$. We let $x$ denote the $k$-adic rational number with base-$k$ expansion $0.x_0x_1\cdots x_d$.  Then for
$j\in \{0,1,\ldots ,s-1\}$, the number $k^{L+s+2+j} x$ has base-$k$ expansion
$$x_0x_{1}\cdots x_{L+s+j+1}. x_{L+s+j+2}\cdots x_d\in [k^{L+s+1},k^{L+2s+1}],$$
and so by Lemma~\ref{lem: hare} there exist $r\le k^{2L+3s+1}$ and $y_1,\ldots ,y_r\in C(u;y,z)$ such that  
$y_1+\cdots +y_r = k^{L+s+2+j} x$.  

Let $\ell$ be a positive integer and let $C_{\ell}(u,v;y,z)$ denote the set of $k$-adic rationals whose base-$k$ expansions are of the form $0.u w_1w_2\cdots w_{\ell} v$ with $w_1,\ldots ,w_{\ell}\in \{y,z\}$ and let $K$ denote the length of $v$.  Observe that given $\epsilon >0$ we have that there is a natural number $N$ such that whenever $x\in C(u;y,z)$ and $\ell>N$ there exists $x'\in C_{\ell}(u,v;y,z)$ such that $|x-x'|<k^{-\ell s - L}$.  In particular, there exist
$y_{1,\ell},y_{2,\ell},\ldots ,y_{r,\ell} \in C_{\ell}(u,v;y,z)$ such that $|y_{i,\ell}-y_i|<k^{-\ell s - L}$ for $i=1,\ldots ,r$.

Thus
$$|y_{1,\ell}+\cdots +y_{r,\ell} - k^{L+s+2+j} x| < rk^{-\ell s -L}\le  k^{2L+3s+1} k^{-\ell s -L} = k^{L+(3-\ell)s+1}.$$
Observe that $k^{L+\ell s + K} y_{i,\ell} \in S$ for $i=1,\ldots ,r$ and so
$k^{L+\ell s+ K}y_{1,\ell} +\cdots + k^{L+\ell s+ K}y_{r,\ell}$ is a sum of at most $k^{2L+3s+1}$ elements of $S$. By construction it is at a distance of at most $k^{L+\ell s + K} k^{L+(3-\ell)s+1} = k^{2L+3s+K+1}$ from
$k^{(\ell+1) s+2L+K+2+j}x$.  Since $j$ can take any value in $\{0,1,\ldots ,s-1\}$ and since $d>K+2L+s+2$, we see that we can find an element in $S^{\le r}$ that is at a distance of at most $k^{2L+3s+K+1}$ from
$M$.  Finally, since $L+2s+1, K+2L+s+2 \le 7m+1$ and $2L+3s+K+1\le 12m+1$, we obtain the desired result.  
\end{proof}

Before proving Theorem \ref{thm:main} 
we need two final results about automatic sets.

\begin{lemma}
Let $k \geq 2$, and suppose $S \subseteq \Enn$ is a $k$-automatic set
and whose minimal accepting DFA has $m$ states.  If $\gcd(S) = 1$ then
there exist distinct integers $s_1, s_2, \ldots, s_{\ell} \in S$, all
less than $k^{2m+2}$, such that
$\gcd(s_1, s_2, \ldots, s_{\ell}) = 1$.
\label{gcd2}
\end{lemma}

\begin{proof}  If $1\in S$, there is nothing to prove, so we may assume that $1\not\in S$.
Let $N$ denote the smallest natural number such that $\gcd(S\cap [1,N+1])=1$ and let $d=\gcd(S\cap [1,N])$.  In particular, $\gcd(d,N+1)=1$. By assumption, $d>1$.  We claim that $N\le k^{2m+2}$.  We write $d=k_0 d_0$, where $\gcd(d_0,k)=1$ and with $k_0$ dividing a power of $k$.  

We first consider the case when $k_0>1$.  Let $a\in \{0,1,\ldots ,k-1\}$ be such that $N+1\equiv \modd{a} {k}$.  Then $\gcd(a,k_0)=1$ since if this is not the case then there is some prime $p$ that divides both $a$, $d$, and $k$ and so $p$ would divide $N+1$ and $d$, which is a contradiction.  Then notice that $S_a:=\{n\ge 0 \colon kn+a\in S\}$ contains $(N+1-a)/k$ and contains no natural number smaller than $(N+1-a)/k$, since if $kn+a\in S$ for some $n< (N+1-a)/k$, then $d|(kn+a)$ and so $k_0|(kn+a)$.  But this is impossible, because if $p$ is a prime that divides $k_0$ (and consequently $k$) then it must divide $a$, which we have shown cannot occur.
Notice that $S_a$ must have a minimal accepting DFA with at most $m$ states.  But it is straightforward to see that a non-empty set whose minimal accepting DFA has at most $m$ states must contain an element of size at most $k^m$ and so 
$N+1 < k^{m+1}+k$.  

Next consider the case when $k_0=1$, so $\gcd(d,k)=1$.  We let $t_s\cdots t_0$ denote the base-$k$ expansion of $N+1$.  We claim that $s\le 2m$.  To see this, suppose that $s>2m$ and let $T_i: = \{n\ge 0\colon k^{i+1}n + [t_i\cdots t_0]_k \in S\}$ for $i=0,\ldots ,m$.
Then since the minimal DFA accepting $S$ has $m$ states we see there exist $i,j\le m$ with $i<j$ such that $T_i=T_j$. 
Also, since each $T_{\ell}$ has a minimal accepting DFA with at most $m$ states and each $T_{\ell}$ is non-empty, we have that there is some least element $r_{\ell} \in T_{\ell}$ with $r_{\ell} < [t_s\cdots t_{\ell+1}]_k \in T_{\ell}$.  Observe that
$r_{\ell}':=k^{\ell+1} r_{\ell}+ [t_{\ell}\cdots t_0]_k<N+1$ and so $d$ divides $r_{\ell}'$.  Moreover, for all $r<[t_s\cdots t_{\ell+1}]_k$ with $r\in T_{\ell}$ we have $k^{\ell+1} r +[t_{\ell}\cdots t_0]_k \equiv \modd{0} {d}$.  Thus since $k$ and $d$ are relatively prime, we see that $T_{\ell}\cap [0,[t_s\cdots t_{\ell+1}]_k-1]$ is non-empty and contained in a single arithmetic progression of difference $d$, but $[t_s\cdots t_{\ell+1}]_k$ is not in this arithmetic progression.  

But now we have that $T_i=T_j$ with $i<j$ and so $T_j\cap [0,[t_s\cdots t_{i+1}]_k-1]$ is contained in a single arithmetic progression mod $d$. 
On the other hand, $T_j\cap  [0,[t_s\cdots t_{j+1}]_k-1]$ is non-empty and contained in a single arithmetic progression mod $d$ and by the above remarks, $[t_s\cdots t_{j+1}]_k<[t_s\cdots t_{i+1}]_k$ is not in this progression, a contradiction.  Thus we see that $s\le 2m$ and so $N < k^{2m+2}$.
\end{proof}

\begin{lemma}
Let $k\ge 2$, $m$ and $c$ be natural numbers and let $S\subseteq \mathbb{N}$ be a $k$-automatic set with $\gcd(S)=1$ and whose minimal accepting DFA has $m$ states.  If $U$ is the set of elements that can be expressed as a sum of at most $2c k^{4m+2}$ elements of $S$ then there is some $N\le c k^{4m+4}$ such that $U$ contains $\{N,N+1,\ldots ,N+c\}$.
\label{lem: gcd}
\end{lemma}

\begin{proof}
From Lemma~\ref{gcd2} we know there exist
$s_1, s_2, \ldots s_{\ell} \in S$ with $s_1<\cdots <s_{\ell}\le k^{2m+2}$ such that
$\gcd(s_1, \ldots, s_{\ell}) = 1$.

It follows from a result of Borosh and Treybig \cite[Theorem 1]{Borosh&Treybig:1976} that there exist integers $a_1,\ldots ,a_{\ell} \in \mathbb{Z}$ with $|a_i|\le k^{2m+2}$ such that $\sum a_i s_i=1$.  

Now let $t=ck^{2m+2}$ and consider the number $N:=t s_1+\cdots + t s_{\ell}$. For each $i=1,\ldots ,c$ we have that
$N+i = (t+i a_1)s_1+\cdots (t+i a_{\ell}) s_{\ell}$ is a nonnegative integer linear combination of $s_1,\ldots ,s_{\ell}$ and $|t+i a_j| \le 2 c k^{2m+2}$ for 
$j\in \{1,\ldots ,\ell\}$.  Thus we see that if $U$ is the set of integers that can be expressed as at most $2 c k^{2m+2} \ell$ elements of $S$, then $U$ contains $\{N,N+1,\ldots ,N+c\}$ where $N=t s_1+\cdots + t s_{\ell} \le c k^{2m+2} \ell$.  Since $\ell\le k^{2m+2}$, we obtain the desired result.     
\end{proof}

We are now ready for the proof of our first main result.

\begin{proof}[Proof of Theorem \ref{thm:main}] Let $m$ be the size of the minimal accepting DFA for $S$.
By Proposition \ref{prop: syndetic} if $T$ is the set of elements that can be expressed as the sum of at most $k^{11m+1}$ elements of $S$ then $T$ is $2k^{12m+1}$-syndetic.  Let $c=2k^{12m+1}$.  By assumption $\gcd(S)=1$ and so by Lemma \ref{lem: gcd} there is some $N_1\le 2c k^{4m+2}=4k^{16m+3}$ and some natural number $M_1\le c k^{4m+4}\le 2k^{16m+5}$ such that each element from $\{M_1,M_1+1,\ldots ,M_1+c\}$ can be expressed as a sum of at most $N_1$ elements of $\{s_1,\ldots ,s_d\}\subseteq S$.  Then let $M_0$ denote the smallest natural number in $T$.  Since $T\supseteq S$ and the minimal DFA for $S$ has size at most $m$, we see that $M_0\le k^m$.  

We claim that every natural number that is greater than $M:=M_0+M_1\le 3k^{16m+5}$ can be expressed as a sum of at most most $N:=k^{11m+1}+N_1\le 5k^{16m+3}$ elements of $S$.  To see this, suppose, in order to get a contradiction, that this is false.  Then there is some smallest natural number $n>M$ that cannot be expressed as a sum of at most $N$ elements of $S$.  Observe that $n-M_1 > M_0$; since $T$ is syndetic and $M_0\in T$, there is some $t\in T$ with $t\le n-M_1 < t+c$.
Thus $n=t+M_1+j$ for some $j\in \{0,1,\ldots ,c-1\}$.  Since $M_1+j$ is a sum of at most $N_1$ elements of $S$ and $t$ is the sum of at most $k^{11m+1}$ elements of $S$, we see that $n$ is the sum of at most $N$ elements of $S$, contradicting our assumption that $n$ has no such representation.  The result follows. 
\end{proof}

\section{An algorithm}
\label{sec:algorithm}
In this section, we prove Theorem~\ref{thm:intro2},
giving an algorithm to find the smallest number $j$ (if it exists) such that $S$ is an asymptotic additive basis (resp., additive
basis) of order $j$
for the natural numbers, where $S$ is a $k$-automatic set of natural numbers.  
We use the fact that there is an algorithm for deciding the truth
of first-order
propositions (involving $+$ and $\leq$) about automatic sequences
\cite{Bruyere&Hansel&Michaux&Villemaire:1994,Allouche&Rampersad&Shallit:2009,Charlier&Rampersad&Shallit:2012}.

\begin{proof}[Proof of Theorem \ref{thm:intro2}]
From Theorem~\ref{thm:main} and Remark~\ref{remark:main}, 
we know that $S$ forms an asymptotic additive basis of order $j$,
for some $j$, if and only if $S$ is non-sparse and has gcd $1$.
This sparsity criterion can be tested using Lemma~\ref{lem1}.  The
condition $\gcd(S)= 1$ can be tested as follows:  compute the smallest
nonzero member $m$ of $S$, if it exists.  Then $\gcd(S)$ must be a divisor
of $m$.  For each divisor $d$ of $m$, form the assertion 
$$ \forall n\geq 0 \ (n \in S) \implies \exists t \text{ such that } n = dt $$
and check it using the algorithm for first-order predicates mentioned
above.  (Note that for each invocation $d$ is actually a constant,
so that $td$ actually is shorthand for
$\overbrace{t + t + \cdots + t}^d$, which uses
addition and not multiplication.)
The largest such $d$ equals $\gcd(S)$.

Once $S$ passes these two tests, we can 
test if $S$ is an asymptotic additive basis of order
$j$ by writing and checking the predicate
\begin{equation}
\exists M \ \forall n \geq M \  \ \exists x_1, x_2, \ldots, x_j \ \text{ such that } x_1, x_2, \ldots , x_j \in S \ \wedge \ n = x_1 + x_2 +  \cdots  + x_j,
\label{opt1}
\end{equation}
which says every sufficiently large integer is the sum of $j$ elements of $S$. 
We do this for $j = 1, 2, 3, \ldots$ until the smallest such $j$ is found.
This algorithm is guaranteed to terminate in
light of Theorem~\ref{thm:main}.

Finally, once $j$ is known, the optimal $M$ in \eqref{opt1} can be determined
as follows by writing the predicate in \eqref{opt1} together
with the assertion that $M$ is the smallest such integer.  Using the
decision procedure mentioned above, one can effectively create a DFA accepting
$(M)_k$, which can then be read off from the transitions of the DFA.

To test if $S$ is an additive basis of order $j$, we need, in addition
to the non-sparseness of $S$ and $\gcd(S) = 1 $, the condition $1 \in S$,
which is easily checked.  
If $S$ passes these tests, we then write and check the predicate
$$
\forall n \geq 0 \  \ \exists x_1, x_2, \ldots, x_j \ \text{ such that } x_1, x_2, \ldots , x_j \in S \ \wedge \ n = x_1 + x_2 +  \cdots  + x_j,
$$
which says every integer is the sum of $j$ elements of $S$.   We do this
for $j = 1,2,3, \ldots$ until the least such $j$ is found.
\end{proof}

\begin{remark}
The same kind of idea can be used to test if every
element of $\Enn$ (or every sufficiently large element)
is the sum of $j$ {\it distinct} elements of a $k$-automatic set $S$.  For
example, if $j = 3$, we would have to add the additional
condition that 
$$x_1 \not= x_2 \ \wedge \ x_1 \not= x_3 \ \wedge \ 
x_2 \not= x_3 .$$

We can also test if every element is {\it uniquely} representable as
a sum of $j$ elements of $S$.   Similarly, we can
count the number $f(n)$ of representations of $n$ as a sum of $j$
elements of $S$.  It follows from 
\cite{Charlier&Rampersad&Shallit:2012} that, for $k$-automatic sets $S$,
the function $f(n)$ is $k$-regular and one can give an explicit
representation for it.
\end{remark}

\section{Examples}
\label{sec:exam}
In this section, we give some examples that illustrate the power of the algorithm provided in the preceding section.  These examples
can be proved ``automatically''
by the {\tt Walnut} theorem-proving software \cite{Mousavi:2016}.

\begin{example}
Let $S$ be the $3$-automatic set of Cantor numbers 
$${\mathcal{C}} = \{0,2,6,8,18,20,24,26,54,56,60,62,72,74,78,80,162, \ldots \},$$ 
that is, those natural numbers
(including $0$) whose base-$3$ expansions consist of only the digits
$0$ and $2$.  Then every even number is the sum of exactly two elements 
of $\mathcal{C}$.  To see this, consider an even natural number $N$. Write $N/2 = x + y$, choosing the base-$3$ expansions of $x$ and $y$ digit-by-digit as follows:
\begin{itemize}
\item[(a)] if the digit of $N/2$ is $2$, choose $1$ for the corresponding digit in both $x$ and $y$;
\item[(b)] if the digit of $N/2$ is $1$, choose $1$ for the corresponding digit in $x$ and $0$ for the corresponding digit in $y$;
\item[(c)] if the digit of $N/2$ is 0, choose $0$ for the corresponding digit in both $x$ and $y$.
\end{itemize}
Then $N = 2x + 2y$ gives the desired representation.
\end{example}

\begin{example}
Let $S$ be the $2$-automatic set of ``evil" numbers 
$$ {\mathcal{E}} = \{ 0,3,5,6,9,10,12,15,17,18,20,23,24,27,29,30,33,34,36,39,
\ldots \},$$
that is, those natural numbers (including $0$) for which the sum of the binary digits is even (see, e.g.,
\cite[p.~431]{Berlekamp&Conway&Guy:1982}).  Then every integer other than $\{1,2,4,7\}$ is the sum
of three elements of ${\mathcal{E}}$.  In fact, every integer except
$\{ 2, 4 \} \ \cup \ \{ 2\cdot 4^i - 1 \ : \ i \geq 1 \}$ is the sum of two elements of
${\mathcal{E}}$.
\end{example}

\begin{example}
Let $S$ be the $2$-automatic set 
$$ {\mathcal{R}} = \{ n \ : r(n) = -1 \} 
=\{3,6,11,12,13,15,19,22,24,25,26,30,35,38,43,44,45,47, \ldots \},$$
where $r(n)$ is the Golay-Rudin-Shapiro function
\cite{Golay:1949,Golay:1951,Rudin:1959,Shapiro:1952}.
Then every integer except
$\{ 0,1,2,3,4,5,7,8,10,11,13,20\}$ is the sum of two elements
of $ {\mathcal{R}}$.
\end{example}

\begin{example}
Let $S$ be the $4$-automatic set 
$$ {\mathcal{D}} = \{ 0,1,4,5,16,17,20,21,64,65,68,69,80,81,84,85, \ldots \}$$
of integers representable in base $4$ using only the digits $0$ and $1$.
See, for example, \cite{Moser:1962,deBruijn:1964}.  Then every natural
is representable as the sum of three elements of ${\mathcal{D}}$.  In fact,
even more is true:  every natural number is uniquely representable as
the sum of one element chosen from ${\mathcal{D}}$ and one element
chosen from $2{\mathcal{D}}$.
\end{example}


\begin{thebibliography}{10}

\bibitem{Allouche&Rampersad&Shallit:2009}
J.-P. Allouche, N.~Rampersad, and J.~Shallit.
\newblock Periodicity, repetitions, and orbits of an automatic sequence.
\newblock {\em Theoret. Comput. Sci.}, 410:2795--2803, 2009.

\bibitem{Allouche&Shallit:2003}
J.-P. Allouche and J.~Shallit.
\newblock {\em Automatic Sequences: Theory, Applications, Generalizations}.
\newblock Cambridge, 2003.

\bibitem{Berlekamp&Conway&Guy:1982}
E.~R. Berlekamp, J.~H. Conway, and R.~K. Guy.
\newblock {\em Winning Ways for your Mathematical Plays}, volume 2: Games in
  Particular.
\newblock Academic Press, 1982.

\bibitem{Borosh&Treybig:1976}
I.~Borosh and L.~B. Treybig.
\newblock Bounds on positive integral solutions of linear {D}iophantine
  equations.
\newblock {\em Proc. Amer. Math. Soc.}, 55(2):299--304, 1976.

\bibitem{deBruijn:1964}
N.~G.~de Bruijn.
\newblock Some direct decompositions of the set of integers.
\newblock {\em Math. Comp.}, 18:537--546, 1964.

\bibitem{Bruyere&Hansel&Michaux&Villemaire:1994}
V.~{Bruy\`ere}, G.~Hansel, C.~Michaux, and R.~Villemaire.
\newblock Logic and $p$-recognizable sets of integers.
\newblock {\em Bull. Belg. Math. Soc.}, 1:191--238, 1994.
\newblock Corrigendum, {\it Bull.\ Belg.\ Math.\ Soc.} 1:577, 1994.

\bibitem{Cabrelli&Hare&Molter:1997}
C.~A. Cabrelli, K.~E. Hare, and U.~M. Molter.
\newblock Sums of {Cantor} sets.
\newblock {\em Ergod. Th. \& Dynam. Sys.}, 17:1299--1313, 1997.

\bibitem{Charlier&Rampersad&Shallit:2012}
E.~Charlier, N.~Rampersad, and J.~Shallit.
\newblock Enumeration and decidable properties of automatic sequences.
\newblock {\em Int. J. Found. Comput. Sci.}, 23:1035--1066, 2012.

\bibitem{Gawrychowski&Krieger&Rampersad&Shallit:2010}
P.~Gawrychowski, D.~Krieger, N.~Rampersad, and J.~Shallit.
\newblock Finding the growth rate of a regular or context-free language in
  polynomial time.
\newblock {\em Int. J. Found. Comput. Sci.}, 21:597--618, 2010.

\bibitem{Ginsburg&Spanier:1966}
S.~Ginsburg and E.~Spanier.
\newblock Bounded regular sets.
\newblock {\em Proc. Amer. Math. Soc.}, 17:1043--1049, 1966.

\bibitem{Golay:1949}
M.~J.~E. Golay.
\newblock Multi-slit spectrometry.
\newblock {\em J. Optical Soc. America}, 39:437--444, 1949.

\bibitem{Golay:1951}
M.~J.~E. Golay.
\newblock Static multislit spectrometry and its application to the panoramic
  display of infrared spectra.
\newblock {\em J. Optical Soc. America}, 41:468--472, 1951.

\bibitem{Hopcroft&Ullman:1979}
J.~E. Hopcroft and J.~D. Ullman.
\newblock {\em Introduction to Automata Theory, Languages, and Computation}.
\newblock Addison-Wesley, 1979.

\bibitem{Ibarra&Ravikumar:1986}
O.~H. Ibarra and B.~Ravikumar.
\newblock On sparseness, ambiguity, and other decision problems for acceptors
  and transducers.
\newblock In B.~Monien and G.~Vidal-Naquet, editors, {\em STACS 86}, volume 210
  of {\em Lect. Notes in Comput. Sci.}, pages 171--179. Springer, 1986.

\bibitem{Kreisel&Lacombe&Shoenfield:1959}
G.~Kreisel, D.~Lacombe, and J.~R. Shoenfield.
\newblock Partial recursive functionals and effective operations.
\newblock In A.~Heyting, editor, {\em Constructivity in Mathematics}, Studies
  in Logic and the Foundations of Mathematics, pages 290--297. North-Holland,
  1959.

\bibitem{Lothaire:1997}
M.~Lothaire.
\newblock {\em Combinatorics on words}.
\newblock Cambridge Mathematical Library. Cambridge University Press,
  Cambridge, 1997.
\newblock With a foreword by Roger Lyndon and a preface by Dominique Perrin,
  Corrected reprint of the 1983 original, with a new preface by Perrin.

\bibitem{Moser:1962}
L.~Moser.
\newblock An application of generating series.
\newblock {\em Math. Mag.}, 35:37--38, 1962.

\bibitem{Mousavi:2016}
H.~Mousavi.
\newblock Automatic theorem proving in {Walnut}.
\newblock Preprint available at \url{https://arxiv.org/abs/1603.06017}, 2016.

\bibitem{Nathanson:1996}
M.~B. Nathanson.
\newblock {\em Additive Number Theory: The Classical Bases}.
\newblock Springer, 1996.

\bibitem{Rudin:1959}
W.~Rudin.
\newblock Some theorems on {Fourier} coefficients.
\newblock {\em Proc. Amer. Math. Soc.}, 10:855--859, 1959.

\bibitem{Shapiro:1952}
H.~S. Shapiro.
\newblock Extremal problems for polynomials and power series.
\newblock Master's thesis, MIT, 1952.

\bibitem{Szilard&Yu&Zhang&Shallit:1992}
A.~Szilard, S.~Yu, K.~Zhang, and J.~Shallit.
\newblock Characterizing regular languages with polynomial densities.
\newblock In I.~M. Havel and V.~Koubek, editors, {\em MFCS 1992}, volume 629 of
  {\em Lect. Notes in Comput. Sci.}, pages 494--503. Springer, 1992.

\bibitem{Trofimov:1981}
V.~I. Trofimov.
\newblock Growth functions of some classes of languages.
\newblock {\em Kibernetika}, 17:9--12, 1981.
\newblock In Russian. English translation in {\it Cybernetics} 17:727--731,
  1981.

\bibitem{Vaughan&Wooley:1991}
R.~C. Vaughan and T.~D. Wooley.
\newblock On {W}aring's problem: some refinements.
\newblock {\em Proc. London Math. Soc. (3)}, 63(1):35--68, 1991.

\bibitem{Wei&Wooley:2015}
B.~Wei and T.~D. Wooley.
\newblock On sums of powers of almost equal primes.
\newblock {\em Proc. Lond. Math. Soc. (3)}, 111(5):1130--1162, 2015.

\bibitem{Wooley:1992}
T.~D. Wooley.
\newblock Large improvements in {W}aring's problem.
\newblock {\em Ann. of Math. (2)}, 135(1):131--164, 1992.

\bibitem{Z:1979}
D.~Zwillinger.
\newblock A {G}oldbach conjecture using twin primes.
\newblock {\em Math. Comp.}, 33:1071, 1979.

\end{thebibliography}
\end{document}